\documentclass[10pt,reqno]{amsart}

\usepackage{amsmath,amssymb,amsfonts,mathrsfs,bm,graphicx,latexsym,verbatim,multicol,lscape,enumitem,makecell,bbm,tabu}

\newtheorem{thm}{Theorem}[section]

\newtheorem{cor}[thm]{Corollary}
\newtheorem{lem}[thm]{Lemma}
\newtheorem{prop}[thm]{Proposition}

\theoremstyle{definition}

\theoremstyle{remark}

\numberwithin{equation}{section}

\newcommand{\rmv}[1]{}

\def\<{\left\langle}
\def\>{\right\rangle}

\begin{document}

\title[On the $2$-adic valuation of $\sigma_k(n)$]{On the $2$-adic valuation of $\sigma_k(n)$}

\author{Kaimin Cheng}
\address{School of Mathematical Sciences, China West Normal University, Nanchong 637002, P. R. China}
\email{ckm20@126.com}

\author{Ke Zhang}
\address{School of Mathematical Sciences, China West Normal University, Nanchong 637002, P. R. China}
\email{2745808109@qq.com}


\date{}

\begin{abstract}
For a positive integer $k$, let
\[
\sigma_k(n)=\sum_{d\mid n} d^k
\]
be the divisor function of order $k$, and let $\nu_p(m)$ denote the $p$-adic valuation of an integer $m$. Motivated by recent work on the $p$-adic valuation of $\sigma_k(n)$, we study $\nu_2(\sigma_k(n))$ in detail. We prove that, for every integer $n\ge 2$,
\[
\nu_2(\sigma_k(n)) \le
\begin{cases}
\lceil \log_2 n \rceil, & \text{if $k$ is odd},\\[1mm]
\lfloor \log_2 n \rfloor, & \text{if $k$ is even}.
\end{cases}
\]
These bounds are best possible. More precisely, if $k$ is odd, then equality holds if and only if $n$ is a product of distinct Mersenne primes; if $k$ is even, then equality holds if and only if $n=3$. We also obtain an explicit formula for $\nu_2(\sigma_k(n))$ in terms of the prime factorization of $n$.
\end{abstract}

\maketitle

\section{Introduction}

For a positive integer $n$, the classical sum-of-divisors function
\[
\sigma(n)=\sum_{d\mid n} d
\]
is one of the most fundamental arithmetic functions in number theory. It appears naturally in many problems, ranging from multiplicative number theory to the theory of perfect numbers. A celebrated theorem of Robin \cite{[Rob84]} states that the Riemann hypothesis is equivalent to the inequality
\[
\sigma(n)<e^\gamma n\log\log n \qquad (n>5041),
\]
where $\gamma$ denotes Euler's constant.

Another classical topic related to $\sigma(n)$ is the theory of perfect numbers. A positive integer $n$ is called \emph{perfect} if $\sigma(n)=2n$. Euler proved that an even integer $n$ is perfect if and only if
\[
n=2^m(2^{m+1}-1)
\]
for some positive integer $m$ such that $2^{m+1}-1$ is prime, that is, a Mersenne prime. Whether there exist infinitely many even perfect numbers remains open, as does the existence of odd perfect numbers. For the latter problem, Ochem and Rao \cite{[Och12]} proved that an odd perfect number must exceed $10^{1500}$, while Nielsen \cite{[Nie15]} showed that an odd perfect number must have at least $10$ distinct prime factors.

Let $p$ be a prime. For a nonzero integer $m$, write $\nu_p(m)$ for the $p$-adic valuation of $m$, namely the largest integer $\alpha$ such that $p^\alpha\mid m$. In recent years, increasing attention has been paid to the $p$-adic valuation of divisor sums. In particular, Amdeberhan, Moll, Sharma, and Villamizar \cite{[Amd21]} studied $\nu_p(\sigma(n))$ and proved that
\[
\nu_2(\sigma(n))\le \lceil \log_2 n\rceil
\]
for every integer $n\ge 2$, with equality if and only if $n$ is a product of distinct Mersenne primes. For odd primes $p$, they also obtained conditional upper bounds for $\nu_p(\sigma(n))$. These conditions were later removed by Zhao and Chen \cite{[Zha25]}, who proved that
\begin{equation}\label{c1-1}
\nu_p(\sigma(n))\le \lceil \log_p n\rceil
\end{equation}
for every odd prime $p$ and every integer $n\ge 2$. They also determined all integers $n\ge 2$ satisfying equality in \eqref{c1-1} for every odd prime $p<10^5$. Related investigations for other arithmetic functions may be found, for instance, in \cite{[Hon12],Qiu19}.

For a positive integer $k$, define the divisor function of order $k$ by
\[
\sigma_k(n)=\sum_{d\mid n} d^k.
\]
Very recently, Zhao \cite{[Zha26]} proved that
\[
\nu_p(\sigma_k(n))\le \lceil k\log_p n\rceil
\]
for every integer $n\ge 2$, every prime $p$, and every integer $k\ge 1$.

The purpose of this paper is to sharpen Zhao's bound in the case $p=2$. Our main result is the following.

\begin{thm}\label{thm1.1}
Let \(k\ge 1\) and \(n\ge 2\). Write
\[
n=2^a\prod_{i=1}^r p_i^{\alpha_i},
\]
where \(p_1,\dots,p_r\) are distinct odd primes. Then
\[
\nu_2(\sigma_k(n))
=
\sum_{\substack{1\le i\le r\\ \alpha_i\ \mathrm{odd}}}
\bigl(\nu_2(\alpha_i+1)+\nu_2(p_i^k+1)-1\bigr).
\]
In particular:
\begin{enumerate}[label=(\alph*), leftmargin=2.5em]
    \item If \(k\) is odd, then
    \[
    \nu_2(\sigma_k(n))
    =
    \sum_{\substack{1\le i\le r\\ \alpha_i\ \mathrm{odd}}}
    \bigl(\nu_2(\alpha_i+1)+\nu_2(p_i+1)-1\bigr),
    \]
    and
    \[
    \nu_2(\sigma_k(n))\le \lceil \log_2 n\rceil.
    \]
    Equality holds if and only if \(n\) is a product of distinct Mersenne primes.

    \item If \(k\) is even, then
    \[
    \nu_2(\sigma_k(n))
    =
    \sum_{\substack{1\le i\le r\\ \alpha_i\ \mathrm{odd}}}
    \nu_2(\alpha_i+1),
    \]
    and
    \[
    \nu_2(\sigma_k(n))\le \lfloor \log_2 n\rfloor.
    \]
    Equality holds if and only if \(n=3\).
\end{enumerate}
\end{thm}

\section{Prime powers}

We begin with the multiplicativity of $\sigma_k$. If $\gcd(m,n)=1$, then
\[
\sigma_k(mn)=\sigma_k(m)\sigma_k(n).
\]
Indeed, every positive divisor $d$ of $mn$ can be written uniquely in the form $d=ab$ with $a\mid m$ and $b\mid n$. Hence
\[
\sigma_k(mn)
=\sum_{d\mid mn}d^k
=\sum_{\substack{a\mid m\\ b\mid n}}(ab)^k
=\left(\sum_{a\mid m}a^k\right)\left(\sum_{b\mid n}b^k\right)
=\sigma_k(m)\sigma_k(n).
\]

Therefore, if
\[
n=2^a\prod_{i=1}^r p_i^{\alpha_i},
\]
where the \(p_i\) are distinct odd primes and \(a,\alpha_i\ge 0\), then
\[
\sigma_k(n)=\sigma_k(2^a)\prod_{i=1}^r \sigma_k(p_i^{\alpha_i}),
\]
and so
\[
\nu_2(\sigma_k(n))
=
\nu_2(\sigma_k(2^a))
+\sum_{i=1}^r \nu_2(\sigma_k(p_i^{\alpha_i})).
\]
Thus the problem reduces to prime powers.

We first consider powers of $2$.

\begin{lem}\label{lem2.1}
For every \(a\ge 0\),
\[
\nu_2(\sigma_k(2^a))=0.
\]
\end{lem}

\begin{proof}
Since
\[
\sigma_k(2^a)=1+2^k+2^{2k}+\cdots+2^{ak},
\]
all terms except the first are even. Hence the sum is odd, and therefore
\[
\nu_2(\sigma_k(2^a))=0.
\]
\end{proof}

Thus the $2$-adic valuation of $\sigma_k(n)$ depends only on the odd part of $n$.

\begin{lem}\label{lem2.2}
If \(n=2^a m\) with \(m\) odd, then
\[
\nu_2(\sigma_k(n))=\nu_2(\sigma_k(m)).
\]
\end{lem}

\begin{proof}
This follows immediately from multiplicativity and Lemma~\ref{lem2.1}.
\end{proof}

The following standard valuation formula is a special case of the lifting-the-exponent lemma.

\begin{lem}\label{lem2.3}\cite[Proposition 1]{[Bey77]}
Let $p$ be a prime and let $A\ge 2$ be an integer. Then for every positive integer $m$:
\begin{enumerate}[label=(\alph*), leftmargin=2.5em]
\item If $p$ is odd and $p\mid (A-1)$, then
\[
\nu_p(A^m-1)=\nu_p(A-1)+\nu_p(m).
\]
\item If $p=2$ and $A$ is odd, then
\[
\nu_2(A^m-1)=
\begin{cases}
\nu_2(A-1), & \text{if $m$ is odd},\\[1mm]
\nu_2(A^2-1)+\nu_2(m)-1, & \text{if $m$ is even}.
\end{cases}
\]
\end{enumerate}
\end{lem}

We now deal with odd prime powers.

\begin{thm}\label{thm2.4}
Let \(p\) be an odd prime and \(\alpha\ge 0\). Then
\[
\nu_2(\sigma_k(p^\alpha))
=
\begin{cases}
0, & \text{if }\alpha\text{ is even},\\[1mm]
\nu_2(\alpha+1)+\nu_2(p^k+1)-1, & \text{if }\alpha\text{ is odd}.
\end{cases}
\]
\end{thm}

\begin{proof}
We have
\[
\sigma_k(p^\alpha)=1+p^k+p^{2k}+\cdots+p^{\alpha k}
=
\frac{p^{k(\alpha+1)}-1}{p^k-1}.
\]

If \(\alpha\) is even, then \(\alpha+1\) is odd. Since \(p\) is odd, each term \(p^{jk}\) is odd, and hence the sum of the \(\alpha+1\) terms is odd. Therefore
\[
\nu_2(\sigma_k(p^\alpha))=0.
\]

Now assume that \(\alpha\) is odd, so that \(\alpha+1\) is even. Put \(A=p^k\). Then \(A\) is odd and
\[
\sigma_k(p^\alpha)=\frac{A^{\alpha+1}-1}{A-1}.
\]
By Lemma~\ref{lem2.3},
\[
\nu_2(A^{\alpha+1}-1)
=
\nu_2(A-1)+\nu_2(A+1)+\nu_2(\alpha+1)-1.
\]
Subtracting \(\nu_2(A-1)\) from both sides yields
\[
\nu_2\!\left(\frac{A^{\alpha+1}-1}{A-1}\right)
=
\nu_2(A+1)+\nu_2(\alpha+1)-1.
\]
Since \(A=p^k\), this becomes
\[
\nu_2(\sigma_k(p^\alpha))
=
\nu_2(p^k+1)+\nu_2(\alpha+1)-1.
\]
This completes the proof.
\end{proof}

The parity of \(k\) leads to two especially simple formulas.

\begin{cor}\label{cor2.5}
Assume that \(k\) is odd. Then for every odd prime \(p\),
\[
\nu_2(p^k+1)=\nu_2(p+1).
\]
Consequently,
\[
\nu_2(\sigma_k(p^\alpha))
=
\begin{cases}
0, & \alpha\ \text{even},\\[1mm]
\nu_2(\alpha+1)+\nu_2(p+1)-1, & \alpha\ \text{odd}.
\end{cases}
\]
\end{cor}

\begin{proof}
Since \(k\) is odd,
\[
p^k+1=(p+1)(p^{k-1}-p^{k-2}+\cdots-p+1).
\]
The second factor is a sum of \(k\) odd integers, hence is itself odd. Therefore
\[
\nu_2(p^k+1)=\nu_2(p+1),
\]
and the desired formula follows from Theorem~\ref{thm2.4}.
\end{proof}

\begin{cor}\label{cor2.6}
Assume that \(k\) is even. Then for every odd prime \(p\),
\[
\nu_2(p^k+1)=1.
\]
Consequently,
\[
\nu_2(\sigma_k(p^\alpha))
=
\begin{cases}
0, & \alpha\ \text{even},\\[1mm]
\nu_2(\alpha+1), & \alpha\ \text{odd}.
\end{cases}
\]
\end{cor}

\begin{proof}
If \(k\) is even and \(p\) is odd, then \(p^2\equiv 1\pmod 8\), so \(p^k\equiv 1\pmod 8\). Hence
\[
p^k+1\equiv 2\pmod 8,
\]
which implies \(\nu_2(p^k+1)=1\). The formula now follows from Theorem~\ref{thm2.4}.
\end{proof}

\begin{thm}\label{thm2.7}
Let
\[
n=2^a\prod_{i=1}^r p_i^{\alpha_i},
\]
where \(p_1,\dots,p_r\) are distinct odd primes. Then
\[
\nu_2(\sigma_k(n))
=
\sum_{\substack{1\le i\le r\\ \alpha_i\ \mathrm{odd}}}
\bigl(\nu_2(\alpha_i+1)+\nu_2(p_i^k+1)-1\bigr).
\]
In particular,
\[
\nu_2(\sigma_k(n))
=
\sum_{\substack{1\le i\le r\\ \alpha_i\ \mathrm{odd}}}
\bigl(\nu_2(\alpha_i+1)+\nu_2(p_i+1)-1\bigr)
\qquad \text{if }k\text{ is odd},
\]
and
\[
\nu_2(\sigma_k(n))
=
\sum_{\substack{1\le i\le r\\ \alpha_i\ \mathrm{odd}}}
\nu_2(\alpha_i+1)
\qquad \text{if }k\text{ is even}.
\]
\end{thm}

\begin{proof}
By multiplicativity,
\[
\nu_2(\sigma_k(n))
=
\nu_2(\sigma_k(2^a))
+\sum_{i=1}^r \nu_2(\sigma_k(p_i^{\alpha_i})).
\]
Now apply Lemma~\ref{lem2.1} and Theorem~\ref{thm2.4}. The two specialized formulas follow from Corollaries~\ref{cor2.5} and \ref{cor2.6}.
\end{proof}

\section{Proof of the main theorem}

We first treat the case where \(k\) is odd.

\begin{prop}\label{prop3.1}
Assume that \(k\) is odd. Then for every \(n\ge 2\),
\[
\nu_2(\sigma_k(n))=\nu_2(\sigma(n)).
\]
\end{prop}

\begin{proof}
By Corollary~\ref{cor2.5}, for each odd prime power \(p^\alpha\),
\[
\nu_2(\sigma_k(p^\alpha))
=
\begin{cases}
0, & \alpha \text{ even},\\[1mm]
\nu_2(\alpha+1)+\nu_2(p+1)-1, & \alpha \text{ odd}.
\end{cases}
\]
This is exactly the same formula as for \(k=1\); see \cite[Theorem 3.2]{[Amd21]}. Moreover, Lemma~\ref{lem2.1} gives
\[
\nu_2(\sigma_k(2^\alpha))=\nu_2(\sigma(2^\alpha))=0.
\]
The claim therefore follows from multiplicativity.
\end{proof}

As an immediate consequence of Proposition~\ref{prop3.1} and \cite[Theorem 1.3]{[Amd21]}, we obtain the following result.

\begin{thm}\label{thm3.2}
Assume that \(k\) is odd. Then for every \(n\ge 2\),
\[
\nu_2(\sigma_k(n))\le \lceil \log_2 n\rceil.
\]
Moreover, equality holds if and only if \(n\) is a product of distinct Mersenne primes.
\end{thm}

We now turn to the case where \(k\) is even.

\begin{lem}\label{lem3.3}
Assume that \(k\) is even, and let
\[
n=\prod_{i=1}^r p_i^{\alpha_i}
\]
be odd. Then
\[
\nu_2(\sigma_k(n))\le \sum_{i=1}^r \alpha_i.
\]
\end{lem}

\begin{proof}
By Theorem~\ref{thm2.7},
\[
\nu_2(\sigma_k(n))
=
\sum_{\alpha_i\ \mathrm{odd}} \nu_2(\alpha_i+1).
\]
If \(\alpha_i\) is odd, then
\[
2^{\nu_2(\alpha_i+1)}\le \alpha_i+1,
\]
and therefore
\[
\nu_2(\alpha_i+1)\le \log_2(\alpha_i+1)\le \alpha_i.
\]
It follows that
\[
\nu_2(\sigma_k(n))
\le
\sum_{\alpha_i\ \mathrm{odd}} \alpha_i
\le
\sum_{i=1}^r \alpha_i.
\]
\end{proof}

The next estimate will be used to isolate the equality case.

\begin{lem}\label{lem3.4}
Let
\[
n=\prod_{i=1}^r p_i^{\alpha_i}
\]
be an odd integer greater than \(3\), and put
\[
\Omega(n)=\sum_{i=1}^r \alpha_i.
\]
Then
\[
\lfloor \log_2 n\rfloor \ge \Omega(n)+1.
\]
\end{lem}

\begin{proof}
We distinguish two cases.

\smallskip
\noindent
\emph{Case 1: \(\Omega(n)=1\).}
Then \(n=p\) is an odd prime. Since \(n>3\), we have \(p\ge 5\), and hence
\[
\lfloor \log_2 n\rfloor \ge \lfloor \log_2 5\rfloor =2=\Omega(n)+1.
\]

\smallskip
\noindent
\emph{Case 2: \(\Omega(n)\ge 2\).}
Since every odd prime factor of \(n\) is at least \(3\), we have
\[
n\ge 3^{\Omega(n)}.
\]
Thus
\[
\log_2 n\ge \Omega(n)\log_2 3.
\]
Since \(\log_2 3>3/2\), it follows that
\[
\log_2 n>\frac32\,\Omega(n).
\]
Because \(\Omega(n)\ge 2\), we have
\[
\frac32\,\Omega(n)\ge \Omega(n)+1.
\]
Therefore
\[
\log_2 n>\Omega(n)+1.
\]
As \(\Omega(n)+1\) is an integer, this implies
\[
\lfloor \log_2 n\rfloor \ge \Omega(n)+1.
\]
\end{proof}

We are now ready to prove the even-\(k\) case of Theorem~\ref{thm1.1}.

\begin{thm}\label{thm3.5}
Assume that \(k\) is even. Then for every \(n\ge 2\),
\[
\nu_2(\sigma_k(n))\le \lfloor \log_2 n\rfloor.
\]
Moreover, equality holds if and only if \(n=3\).
\end{thm}

\begin{proof}
Write
\[
n=2^a m,
\]
where \(m\) is odd. By Lemma~\ref{lem2.2},
\[
\nu_2(\sigma_k(n))=\nu_2(\sigma_k(m)).
\]

We first prove the upper bound. If \(m=1\), then \(n=2^a\), and Lemma~\ref{lem2.1} gives \(\nu_2(\sigma_k(n))=0\), so the assertion is clear. Thus we may assume that \(m>1\), and write
\[
m=\prod_{i=1}^r p_i^{\alpha_i}.
\]
By Lemma~\ref{lem3.3},
\[
\nu_2(\sigma_k(n))
=
\nu_2(\sigma_k(m))
\le
\sum_{i=1}^r \alpha_i.
\]
On the other hand,
\[
\log_2 n \ge \log_2 m = \sum_{i=1}^r \alpha_i \log_2 p_i \ge \sum_{i=1}^r \alpha_i,
\]
since each odd prime \(p_i\ge 3\) satisfies \(\log_2 p_i>1\). Hence
\[
\nu_2(\sigma_k(n))\le \log_2 n.
\]
Since the left-hand side is an integer, it follows that
\[
\nu_2(\sigma_k(n))\le \lfloor \log_2 n\rfloor.
\]

We now consider the equality case. Suppose that
\[
\nu_2(\sigma_k(n))=\lfloor \log_2 n\rfloor.
\]
We first show that \(a=0\), so that \(n\) must be odd. Indeed, if \(a\ge 1\), then
\[
\log_2 n=\log_2 m+a,
\]
and therefore
\[
\lfloor \log_2 n\rfloor \ge \lfloor \log_2 m\rfloor+1.
\]
But
\[
\nu_2(\sigma_k(n))=\nu_2(\sigma_k(m)),
\]
so equality for \(n\) would imply
\[
\nu_2(\sigma_k(m))\ge \lfloor \log_2 m\rfloor+1,
\]
contrary to the bound already established for \(m\). Hence \(a=0\).

Thus \(n\) is odd. If \(n>3\), then Lemma~\ref{lem3.4} yields
\[
\lfloor \log_2 n\rfloor \ge \Omega(n)+1,
\]
where \(\Omega(n)=\sum_i \alpha_i\), while Lemma~\ref{lem3.3} gives
\[
\nu_2(\sigma_k(n))\le \Omega(n).
\]
Therefore
\[
\nu_2(\sigma_k(n))<\lfloor \log_2 n\rfloor,
\]
so equality is impossible for odd \(n>3\). It remains only to consider \(n=3\).

For \(n=3\), we have
\[
\sigma_k(3)=1+3^k.
\]
Since \(k\) is even, \(3^k\equiv 1\pmod 8\), and hence
\[
1+3^k\equiv 2\pmod 8.
\]
Thus
\[
\nu_2(\sigma_k(3))=1=\lfloor \log_2 3\rfloor.
\]
This proves that equality holds if and only if \(n=3\).
\end{proof}

Theorem~\ref{thm1.1} follows immediately from Theorems~\ref{thm3.2} and \ref{thm3.5}.

\end{document}